\documentclass[11pt,reqno]{amsart}

\usepackage{tikz}
\usetikzlibrary{calc}
\usepackage[T1]{fontenc}
\usepackage{lmodern}
\usepackage[a4paper,margin=28mm]{geometry}
\usepackage{amsmath,amssymb,amsthm,mathtools}
\usepackage[expansion=false]{microtype}
\usepackage[colorlinks=true,linkcolor=blue!50!black,
  citecolor=blue!50!black,urlcolor=blue!50!black]{hyperref}
\usepackage{xcolor}
\usepackage{subfig}
\usepackage{caption}
\hypersetup{
 pdftitle={A geometric characterization of the MS limit},
 pdfauthor={},
 pdfkeywords={}
}

\newtheorem{theorem}{Theorem}[section]

\newtheorem{lemma}[theorem]{Lemma}

\theoremstyle{definition}

\theoremstyle{remark}
\newtheorem{remark}[theorem]{Remark}

\newcommand{\dd}{\,\mathrm{d}}
\newcommand{\R}{\mathbb R}
\newcommand{\D}{\mathcal D}
\newcommand{\E}{\mathcal E}
\newcommand{\F}{\mathcal F}
\newcommand{\norm}[1]{\lVert#1\rVert}
\newcommand{\MS}{Maz'ya--Shaposhnikova}
\newcommand{\av}{\mathsf M}
\newcommand{\doi}[1]{\href{https://doi.org/#1}{\nolinkurl{doi:#1}}}

\numberwithin{equation}{section}
\allowdisplaybreaks[2]
\AtEndDocument{
	\bigskip{\footnotesize
		\textsc{School of Mathematics, Nanjing University, Nanjing 210093, China.}
		\par
		\textsc{School of Mathematics, Hangzhou Normal University, Hangzhou 310036, China}  \par
		\textit{E-mail address}: \texttt{gaojin@hznu.edu.cn} \par
	}
	\bigskip{\footnotesize
		\textsc{College of Science, National University of Defense Technology, Changsha 410073, China}  \par
		\textit{E-mail address}: \texttt{yuzy23@nudt.edu.cn} \par
	}
	\bigskip{\footnotesize
		\textsc{School of Mathematics, South China University of Technology, Guangzhou 510641, China.}  \par
		\textit{E-mail address}: \texttt{summerfish@scut.edu.cn} \par
	}
}

\title[A geometric characterization of the MS limit]
{A geometric characterization of the Maz'ya-Shaposhnikova limit on metric measure spaces}
\author{Jin Gao, Zhenyu Yu and Junda Zhang}
\date{}


\begin{document}

\begin{abstract}
We provide a purely geometric necessary and sufficient condition, in terms of a certain Ces\`{a}ro mean of the volume function, for the existence of the Maz'ya--Shaposhnikova functional limit with respect to
Gagliardo seminorms on metric measure spaces (under a mild measure condition). As an application, we establish the Maz'ya--Shaposhnikova convergence on the unbounded Sierpi\'nski gasket. We also show that Ahlfors regularity does not imply the existence of this functional limit by constructing a concrete example.
\end{abstract}

\maketitle
\enlargethispage{3pt}

\section{Introduction}\label{sec:intro}
Recall the classical `Bourgain-Brezis-Mironescu (BBM) convergence' in \cite[%
Corollary 2]{BBM}:
\begin{equation}
\lim_{s \uparrow 1}(1-s )\int_{D}\int_{D}\frac{|u(x)-u(y)|^{p}}{%
|x-y|^{n+ps }}\dd x\dd y=C_{n,p}\int_{D}|\nabla u(x)|^{p}\dd x \ \ \
(1<p<\infty),  \label{BBM}
\end{equation}
where $D$ is a smooth domain in $\mathbb{R}^{n}$.
We refer the reader to \cite{BSY23} for the BBM convergence on Triebel--Lizorkin spaces, to \cite{DGPYYZ24} in the ball Banach function space setting, to \cite{PP08,Yang,GL20A,GYZ22,GYZ23} in the fractal setting, and to \cite{GYZ26,LPZ24,Shi25} in the general metric measure space setting and the references therein.
Parallel to the BBM convergence \cite{BBM},
Maz'ya and Shaposhnikova \cite{MS} proved that
for $1\le p<\infty$,
\begin{equation}\label{eq:euclidean}
 \lim_{s\downarrow0}s\iint_{\R^n\times\R^n}
 \frac{|u(x)-u(y)|^p}{|x-y|^{n+sp}}\dd x\dd y
 =\frac{2|S^{n-1}|}{p}\norm{u}_{L^p(\R^n)}^p
\end{equation}
for any $u$ with finite positive-order Gagliardo seminorm, where $|S^{n-1}|$ is the area of the unit sphere $S^{n-1}$.
The BBM convergence describes the behaviour of the Gagliardo seminorm (with scaling factor $1-s$) as $s\uparrow 1$, while the Maz'ya--Shaposhnikova (MS) convergence describes the opposite case $s\downarrow0$ (with scaling factor $s$).

The MS convergence has been developed in several directions. A sufficient condition for the above MS convergence on metric measure spaces including unbounded fractals is given in \cite[Theorem~2.2]{HPXZ}, where the above kernel $|x-y|^{-(n+sp)}$ is also generalised to a wide family of mollifiers.
Han \cite{Han} described the geometric structure of the Bourgain--Brezis--Mironescu (BBM) and MS convergence on certain metric measure spaces with gradient structures, where the kernels are general mollifiers. Necessary and sufficient conditions on kernels $\rho_{s}(x-y)$ for the existence of MS limit have also
been investigated in \cite{DDFGP}, where the underlying metric measure space is $\R^n$ (with the standard Lebesgue measure). Gu and Yung \cite{GuYang} switched the Gagliardo seminorm in the left-hand side of \eqref{eq:euclidean} to a certain weak-$L^p$ quasi-norm (the measures of suitable level sets), while the right-hand side remained the $L^p$ norm.

We extend the study of the MS convergence in the following direction.
We focus on the specific classical kernel $d(x,y)^{-(\alpha+sp)}$ for metric measure spaces satisfying a mild measure condition, including a wide range of fractals. The main result is a purely geometric necessary and sufficient condition for the MS convergence, in terms of logarithmic-type Ces\`{a}ro mean of the volume function. We also show that the classical Ahlfors regularity condition does not guarantee such convergence.

Let us precisely state our setting and the main result. Let $(X,d)$ be an unbounded complete separable metric space.  Let
\[
 B_r(x):=\{y\in X:d(x,y)\le r\}
\]
be closed balls throughout this paper. Let $\mu$ be a Borel measure satisfying
\begin{equation}\label{eq:upper-growth}
 0< V_x(r):=V(x,r)=\mu(B_r(x))\le C_Vr^{\alpha}
 \qquad (x\in X,\ r\ge1),
\end{equation}
where $\alpha$ and $C_V$ are two positive constants.

Denote the norm in $L^p:=L^p(X,\mu)\ (1\leq p<\infty)$ by
\begin{equation*}
||u||_p:=\left(\int_X|u(x)|^p \dd \mu(x)\right)^{1/p},
\end{equation*}
and $||u||_{\infty}:=\mathrm{ess\,sup}_{x\in X}|u(x)|$, where $\mathrm{ess\,sup}$ is the essential
supremum.

For $s>0$ and $1\le p<\infty$, define the Gagliardo seminorm with domain
\begin{align}
 [u]_{s,p;\alpha}^p
  &=\iint_{x\ne y}\frac{|u(x)-u(y)|^p}{d(x,y)^{\alpha+sp}}
       \dd\mu(x)\dd\mu(y),\label{eq:seminorm}\\
 \E_s(u)&=s[u]_{s,p;\alpha}^p,\label{eq:energy}\\
 \D_{\alpha,p}
  &=\bigcup_{0<s_0<1}
    \{u\in L^p(X,\mu):[u]_{s_0,p;\alpha}<\infty\}.\label{eq:class}
\end{align}
The major parts, namely the large-scale Gagliardo seminorm and kernel, are denoted by
\begin{align}
 \F_s(u)&=s\iint_{d(x,y)>1}
   \frac{|u(x)-u(y)|^p}{d(x,y)^{\alpha+sp}}
       \dd\mu(x)\dd\mu(y),\label{eq:far-energy}\\
 H_s(x)&=s\int_{d(x,y)>1}d(x,y)^{-\alpha-sp}\dd\mu(y).
       \label{eq:scalar-tail}
\end{align}

Our main result is the following theorem.
\begin{theorem}
\label{thm:main}
Assume \eqref{eq:upper-growth}. Fix $1\le p<\infty$ and $o\in X$. Then the following statements hold.
\begin{enumerate}
\item For every $u\in L^p(X,\mu)$,
\begin{equation}\label{eq:far-reduction}
 \F_s(u)-2H_s(o)\norm{u}_p^p\longrightarrow0
 \qquad \text{as }s\downarrow0.
\end{equation}

\item For every $u\in\D_{\alpha,p}$,\begin{equation}
 \E_s(u)-2H_s(o)\norm{u}_p^p\longrightarrow0
 \qquad \text{as }s\downarrow0.
\end{equation}

\item Let $h_o(t)=e^{-\alpha t}V_o(e^t)$ and $M_o(T)=\frac1T\int_0^T h_o(t)\dd t$. For every function
$u\in\D_{\alpha,p}$ with positive $L^p$-norm, the following three conditions are equivalent:

(a) $\lim_{s\downarrow0}\E_s(u)$ exists;

(b) $\lim_{s\downarrow0}H_s(o)$ exists;

(c) $\lim_{T\to\infty}M_o(T):=\av_{\alpha}(X)$ exists.

Whenever these conditions hold,
\begin{equation}\label{eq:logarithmic-mean}
 \av_{\alpha}(X)=\lim_{R\to\infty}\frac1{\log R}
 \int_1^R\frac{V_o(r)}{r^{\alpha}}\frac{\dd r}{r}
\end{equation}
is independent of the choice of the base
point $o$, and for every $v\in\D_{\alpha,p}$,
\begin{equation}\label{eq:main-limit}
 \lim_{s\downarrow0}\E_s(v)
 =\frac{2\alpha}{p}\av_{\alpha}(X)\norm{v}_p^p.
\end{equation}
\end{enumerate}
\end{theorem}

We remark that only unbounded spaces are of interest, since it is pointed out in \cite[Remark 2.3]{HPXZ} that $\lim_{s\downarrow0}\E_s(u)=0$ for bounded spaces (with bounded total measure).

Let us briefly outline the proof strategy. It is natural to consider first, functions with bounded support, and then a standard approximation will give the unbounded cases. In view of the previous studies, it is also natural to consider the large-scale influence as $s$ tends to 0, and we therefore introduce the first two inclusions, which hold in general. One may take a bump function on the Euclidean space as an intuitive example for the first inclusion. In the third inclusion, a certain Hardy-Littlewood-Karamata Tauberian result will be used to obtain \eqref{eq:logarithmic-mean}. To show that the limit is independent of the choice of the base point, we establish a substitution comparison (Lemma \ref{lem:tails}), which turns out to be very useful.

We say that $(X,d,\mu)$ is \emph{$\alpha$-Ahlfors regular} if there
are constants $0<c\le C<\infty$ such that
\[
 cr^{\alpha}\le\mu(B_r(x))\le Cr^{\alpha}
 \qquad(x\in X,\ r>0).
\]
In Theorem \ref{thm:counterexample}, we show that Ahlfors regularity does not imply the existence of the above limits. The idea is that Ahlfors regularity allows the constants $c$ and $C$ to vary, thus a careful choice of such variation will lead to great fluctuations in the Ces\`{a}ro mean in Theorem \ref{thm:main}.

The organization is as follows. We prove Theorem \ref{thm:main} in
Section~\ref{sec:proof}.  In Section \ref{sec:counterexample}, we prove Theorem \ref{thm:counterexample}. In
Section~\ref{sec:sg-verification}, we apply part (3) of Theorem \ref{thm:main} on the unbounded Sierpi\'nski gasket to establish the MS convergence result.

\textbf{Notation}: the letters $C$,$C^{\prime }$,$C^{\prime \prime }$ and $c$ are universal positive constants which may vary at each occurrence. The subscripts in the above notation indicate the variables on which the constant depends.

\section{Proof of Theorem \ref{thm:main}}\label{sec:proof}

The following lemma gives the uniform boundedness of
$H_s(\cdot)$ and the difference estimate for
$H_s(\cdot)$ at nearby base points for $0<s\le1$, which are essentially used in parts (1) and (3) of the proof of Theorem \ref{thm:main}.

\begin{lemma}\label{lem:tails}
For $0<s\le1$ and $R\ge1$, the assumption \eqref{eq:upper-growth} implies
\begin{equation}\label{eq:tail-bound}
 \sup_{x\in X}s\int_{d(x,y)>R}d(x,y)^{-\alpha-sp}\dd\mu(y)
 \le C_V\frac{\alpha+sp}{p}R^{-sp}.
\end{equation}
Moreover, for every $M\ge1$, there is a constant $C_M$ independent of
$s\in(0,1]$, such that
\begin{equation}\label{eq:base-point}
 \sup_{d(o,x)\le M}|H_s(x)-H_s(o)|\le C_Ms.
\end{equation}

\end{lemma}

\begin{proof}
Fix $0<s\le1$. Tonelli's theorem and
\eqref{eq:upper-growth} imply that
\begin{align*}
 \int_{d(x,y)>R}d(x,y)^{-\alpha-sp}\dd\mu(y)&= \int_{d(x,y)>R}(\alpha+sp)\int_{d(x,y)}^\infty t^{-\alpha-sp-1}\dd t\dd\mu(y) \\
 &=(\alpha+sp)\int_R^\infty t^{-\alpha-sp-1}
   \mu\{y:R<d(x,y)\le t\}\dd t\\
 &\le (\alpha+sp)C_V\int_R^\infty t^{-sp-1}\dd t\\
 &
 =\frac{(\alpha+sp)C_V}{sp}R^{-sp}\leq\frac{(\alpha+sp)C_V}{p}R^{-sp} ,
\end{align*}
showing \eqref{eq:tail-bound}. Similarly, we will later use the following
\begin{align}
 \int_{d(o,y)>R}d(o,y)^{-\alpha-sp-1}\dd\mu(y)\le C_V\frac{\alpha+sp+1}{sp+1}R^{-sp-1}\le \frac{C_V(\alpha+1)}{R}.
 \label{eq:shift-integrability}
\end{align}

To see the second inclusion, we fix $M\ge1$ and choose $R>2M+2$.  If $d(o,x)\le M$ and $d(o,y)>R$, the triangle inequality yields
\begin{align*}
|d(x,y)-d(o,y)|\le M \leq \frac{d(o,y)}{2}-1.
\end{align*}
Hence, we have
\begin{align*}
\frac{d(o,y)}{2} \le d(x,y)\le \frac{3d(o,y)}{2}.
\end{align*}
Applying the mean value theorem to $t^{-\alpha-sp}$ yields that, for some number $\xi$
between $d(x,y)$ and $d(o,y)$,
\begin{align}
 |d(x,y)^{-\alpha-sp}-d(o,y)^{-\alpha-sp}|
 &=(\alpha+sp)\xi^{-\alpha-sp-1}|d(x,y)-d(o,y)|\notag\\
 &\le (\alpha+p)2^{\alpha+p+1}M d(o,y)^{-\alpha-sp-1},
 \label{eq:kernel-shift}
\end{align}
where we used $0<s\le 1$ and $\xi\ge d(o,y)/2$.  Combining
\eqref{eq:kernel-shift} with \eqref{eq:shift-integrability}, we obtain
\begin{equation}\label{eq:outer-shift}
 \int_{X\setminus B_R(o)}
 |d(x,y)^{-\alpha-sp}-d(o,y)^{-\alpha-sp}|\dd\mu(y)
 \le C_{\alpha,p,M,R}
\end{equation}
uniformly for $0<s\le1$ and $d(o,x)\le M$. Define $K_{s,z}(y)=\mathbf 1_{\{d(z,y)>1\}}d(z,y)^{-\alpha-sp}$. Then,
\begin{align*}
 |H_s(x)-H_s(o)|
 &\le s\int_{B_R(o)}|K_{s,x}(y)-K_{s,o}(y)|\dd\mu(y)\\
 &\quad+s\int_{X\setminus B_R(o)}
       |K_{s,x}(y)-K_{s,o}(y)|\dd\mu(y)\\
 &\le 2s\mu(B_R(o))
   +s\int_{X\setminus B_R(o)}
       |d(x,y)^{-\alpha-sp}-d(o,y)^{-\alpha-sp}|\dd\mu(y)\\
 &\le \bigl(2C_Vr^{\alpha}+C_{\alpha,p,M,R}\bigr)s,
\end{align*}
where we used the fact that $0\le K_{s,z}\le1$ on $B_R(o)$ and \eqref{eq:outer-shift}.  The proof is now complete.
\end{proof}

We are now ready to prove the first two statements in Theorem~\ref{thm:main}.

\begin{proof}[Proof of Theorem~\ref{thm:main} (1) and (2)]
We first consider a function $u$ supported on
$B_M(o)$, where $M\ge1$.  Fix $R>2M+2$.
By symmetry, we can decompose
\begin{align}
 \F_s(u)
 ={}&s\iint_{\substack{x,y\in B_R(o)\\d(x,y)>1}}
 |u(x)-u(y)|^p d(x,y)^{-\alpha-sp}\dd\mu(x)\dd\mu(y)\notag\\
 &+2s\int_{B_M(o)}|u(x)|^p
   \int_{d(o,y)>R}d(x,y)^{-\alpha-sp}\dd\mu(y)\dd\mu(x).
 \label{eq:compact-split}
\end{align}

For the first term in \eqref{eq:compact-split}, since
$|u(x)-u(y)|^p\le2^{p-1}(|u(x)|^p+|u(y)|^p)$, we have
\begin{equation}\label{eq:bounded-piece}
s\iint_{\substack{x,y\in B_R(o)\\d(x,y)>1}}
|u(x)-u(y)|^p d(x,y)^{-\alpha-sp}\dd\mu(x)\dd\mu(y) \leq 2^ps\mu(B_R(o))\norm{u}_p^p.
\end{equation}

We need to estimate the second term in \eqref{eq:compact-split}. For $x\in B_M(o)$, define
\begin{align*}
 T_s(x)=s\int_{d(o,y)>R}d(x,y)^{-\alpha-sp}\dd\mu(y).
\end{align*}
Clearly,
\begin{align*}
 T_s(x)-H_s(o)
 ={}&s\int_{d(o,y)>R}
 \bigl(d(x,y)^{-\alpha-sp}-d(o,y)^{-\alpha-sp}\bigr)\dd\mu(y)\\
 &-s\int_{1<d(o,y)\le R}d(o,y)^{-\alpha-sp}\dd\mu(y).
\end{align*}
The absolute value of the first term above is at most $C_{\alpha,p,M,R}\cdot s$, by
\eqref{eq:outer-shift}, while that of the second term is at most $s\mu(B_R(o))$.  Therefore
\begin{equation}\label{eq:T-uniform}
 \sup_{x\in B_M(o)}|T_s(x)-H_s(o)|\le C'_{\alpha,p,M,R}s.
\end{equation}

Substituting the above estimate \eqref{eq:T-uniform} into \eqref{eq:compact-split} and using
\eqref{eq:bounded-piece}, we have
\begin{equation}\label{eq:compact-rate}
 |\F_s(u)-2H_s(o)\norm{u}_p^p|
 \le C''_{\alpha,p,M,R}s\norm{u}_p^p.
\end{equation}

Next, we consider general functions $u$.  We introduce the measure
\[
 \dd\lambda_s(x,y)=s\mathbf 1_{\{d(x,y)>1\}}
 d(x,y)^{-\alpha-sp}\dd\mu(x)\dd\mu(y).
\]
We write $\F_s(u)^{1/p}=\|u(x)-u(y)\|_{L^p(\lambda_s)}$, and use
Minkowski's inequality to obtain
\begin{equation}\label{eq:Lp-control}
 |\F_s(u)^{1/p}-\F_s(v)^{1/p}|
 \le\F_s(u-v)^{1/p}\le C\norm{u-v}_p,
 \qquad0<s\le1.
\end{equation}
To see the second inequality above, we note that
\begin{align}
 \F_s(z)
 &\le 2^{p-1}s\iint_{d(x,y)>1}
   \frac{|z(x)|^p+|z(y)|^p}{d(x,y)^{\alpha+sp}}
        \dd\mu(x)\dd\mu(y)\notag\\
 &=2^p\int_X|z(x)|^p
   \left(s\int_{d(x,y)>1}d(x,y)^{-\alpha-sp}\dd\mu(y)\right)\dd\mu(x)\notag\\
 &\le 2^p C_V\frac{\alpha+sp}{p}\norm{z}_p^p
 \le 2^p C_V\frac{\alpha+p}{p}\norm{z}_p^p,\label{eq_Fs}
\end{align}
where we used \eqref{eq:tail-bound} with $R=1$ in the second inequality.

Let $v_M=u\mathbf1_{B_M(o)}$.  Using the fact
$|A^{1/p}-B^{1/p}|\le|A-B|^{1/p}$ for $A,B\ge0$ and
applying \eqref{eq:compact-rate}, we obtain
\begin{equation}\label{eq:compact-root}
 \left|\F_s(v_M)^{1/p}-(2H_s(o))^{1/p}\norm{v_M}_p\right|
 \longrightarrow0
 \qquad \text{as }s\downarrow0
\end{equation}
for every fixed $M$.  Applying the triangle inequality, \eqref{eq:Lp-control},
and $|\|u\|_p-\|v_M\|_p|\le\|u-v_M\|_p$, we obtain
\begin{align*}
\left|\F_s(u)^{1/p}-(2H_s(o))^{1/p}\norm{u}_p\right|\le C\norm{u-v_M}_p
 +(2H_s(o))^{1/p}\norm{u-v_M}_p \\
 +\left|\F_s(v_M)^{1/p}
        -(2H_s(o))^{1/p}\norm{v_M}_p\right|.
\end{align*}
It follows that
\[
 \F_s(u)^{1/p}-(2H_s(o))^{1/p}\norm{u}_p\longrightarrow0
\]
 by letting $M\to\infty$ and using
\eqref{eq:compact-root} with
$\|u-v_M\|_p\to0$.
Note that $H_s(o)$ is bounded uniformly for
$0<s\le1$  by
\eqref{eq:tail-bound}, and also $\F_s(u)$ by \eqref{eq_Fs}.
On a uniformly bounded interval the map $t\mapsto t^p$ is Lipschitz, showing Theorem~\ref{thm:main} (1).

Finally, suppose $[u]_{s_0,p;\alpha}<\infty$ for some $0<s_0<1$.  If
$0<s\le s_0$, then we have
\begin{align}
 0\le\E_s(u)-\F_s(u)=s\iint_{0<d(x,y)\le1}
 \frac{|u(x)-u(y)|^p}{d(x,y)^{\alpha+sp}} \dd\mu(x)\dd\mu(y)\le s[u]_{s_0,p;\alpha}^p\label{eq:near-part}
\end{align}
by using the elementary fact $t^{-\alpha-sp}\le t^{-\alpha-s_0p}$ for $0<t\le1$.  The right-hand side
tends to zero, and thus Theorem~\ref{thm:main} (2) follows from (1).
\end{proof}

We will use the following Hardy-Littlewood-Karamata Tauberian result for Laplace-Stieltjes transforms.

\begin{lemma}\label{lem:abel-cesaro}
Let $h$ be a bounded measurable real-valued function on $[0,\infty)$, and let
\begin{equation}\label{e.Abel}
  A_\varepsilon(h)=\varepsilon\int_0^\infty
 e^{-\varepsilon t}h(t)\dd t,\qquad
 M_h(T)=\frac1T\int_0^T h(t)\dd t.
\end{equation}
Then $A_\varepsilon(h)\to m$ as $\varepsilon\downarrow0$ if and only if
$M_h(T)\to m$ as $T\to\infty$.
\end{lemma}

\begin{proof}

Let $F(t)=\int_0^t h(v)\dd v$. Then $M_h(t)=F(t)/t\leq \|h\|_\infty$. If $M_h(t)\to m$ as $t\to \infty$, integration by parts
and the substitution $v=\varepsilon t$ give
\[
 A_\varepsilon(h)=\varepsilon^2\int_0^\infty
 e^{-\varepsilon t}F(t)\dd t
 =\int_0^\infty ve^{-v}M_h(v/\varepsilon)\dd v\longrightarrow m,
\]
as $\varepsilon\downarrow0$ by the dominated convergence theorem, since the
dominant function is $ve^{-v}\|h\|_\infty$, which belongs to $L^1((0,\infty))$.

Conversely, assume that $A_\varepsilon(h)\to m$ as $\varepsilon\downarrow0$. For each integer $k\ge1$, the substitution $v=Tt$ gives
\[
\int_0^\infty e^{-kt}h(Tt)\dd t
=\frac1k A_{k/T}(h)\longrightarrow\frac{m}{k}
=m\int_0^\infty e^{-kt}\dd t,
\]
as $T\to\infty$.
We will use the standard result that the linear span of $\{e^{-kt}:k\ge1\}$ over $\R$ is dense in
$L^1(0,\infty)$. Indeed, for $f\in C_c([0,\infty))$, define
$g(x)=f(-\log x)/x$ for $0<x\le1$ and $g(0)=0$. Since $f$ has compact support, $g$ is continuous on $[0,1]$. If a polynomial $P$ approximates $g$ uniformly on $[0,1]$, then
\[
\int_0^\infty|f(t)-e^{-t}P(e^{-t})|\dd t
=\int_0^1|g(x)-P(x)|\dd x.
\]
The function $e^{-t}P(e^{-t})$ is a finite linear combination of $e^{-kt}$, $k\ge1$, and the standard result then follows from the Weierstrass approximation and the fact that $C_c([0,\infty))$ is dense in $L^1(0,\infty)$. Since $|h(Tt)|\le\|h\|_\infty$ uniformly for all $T$, we have
\[
M_h(T)=\int_0^1h(Tt)\dd t
=\int_0^\infty\mathbf1_{[0,1]}(t)h(Tt)\dd t
\longrightarrow m\int_0^\infty\mathbf1_{[0,1]}(t)\dd t=m,
\]
showing the desired result.
\end{proof}

We turn to the last part of the proof of Theorem~\ref{thm:main}.

\begin{proof}[Proof of Theorem~\ref{thm:main} (3)]
For $R>1$, Stieltjes integration by parts on $(1,R]$ gives
\[
 \int_{(1,R]}r^{-\alpha-sp}\dd V_o(r)
 =R^{-\alpha-sp}V_o(R)-V_o(1)
  +(\alpha+sp)\int_1^R V_o(r)r^{-\alpha-sp-1}\dd r.
\]
By
\eqref{eq:upper-growth}, $R^{-\alpha-sp}V_o(R)\le C_VR^{-sp}\to0$ as
$R\to\infty$.  Therefore, letting $d(o,y)=r$,
\begin{align}
 H_s(o)
 &=s\int_{(1,\infty)}r^{-\alpha-sp}\dd V_o(r)\notag\\
 &=-sV_o(1)+s(\alpha+sp)\int_1^\infty
      V_o(r)r^{-\alpha-sp-1}\dd r\notag\\
 &=-sV_o(1)+\frac{\alpha+sp}{p}A_{sp}(h_o).
 \label{eq:stieltjes-identity}
\end{align}
where in the last equality we substituted $r=e^t$ and used the definition of $A_{sp}(h_o)$.
Since $h_o(t)=e^{-\alpha t}V_o(e^t)\le C_V$ is bounded,
Lemma~\ref{lem:abel-cesaro} shows that
\begin{equation}\label{e.equi}
M_o(T)\to m \ (T\to\infty ) \Longleftrightarrow \lim_{s\downarrow0}H_s(o)=\frac{\alpha}{p}m,
\end{equation}
which implies the equivalence of conditions (b) and (c) in part (3).

For a fixed function $u$ in part~(3), part~(2) gives
\begin{equation}\label{e.equi2}
  \lim_{s\downarrow0}\left(\E_s(u)-2H_s(o)\norm{u}_p^p\right)=0.
\end{equation}
Since $\|u\|_p^p>0$, this shows the equivalence of conditions (a) and (b) of part (3). Thus, conditions (a), (b), (c) of part (3) are equivalent.

If $\lim_{T\to\infty}M_o(T):=\av_{\alpha}(X)$ exists, then \eqref{eq:logarithmic-mean} follows from direct computation and the existence of $\lim_{s\downarrow0}H_s(o)$. Let $x$ be another chosen base point. Lemma~\ref{lem:tails} together with $M\ge\max\{1,d(o,x)\}$ gives $ \lim_{s\downarrow0}H_s(x)= \lim_{s\downarrow0}H_s(o)$. Applying \eqref{e.equi} at point $x$ shows that the value $\av_{\alpha}(X)$ is independent of the base point.

 Applying \eqref{e.equi} and \eqref{e.equi2} to $v\in\D_{\alpha,p}$ gives \eqref{eq:main-limit}, which completes the proof.
\end{proof}

\section{Ahlfors regularity does not imply the MS convergence}
\label{sec:counterexample}

We present a precise construction by choosing a suitable measure on the half-line to show that Ahlfors regularity does not guarantee the existence of the logarithmic-volume mean or the
\MS{} limit.

The intuition behind our construction is as follows. As we have mentioned, Ahlfors regularity allows the variation of the density constant, and such variation could influence the Ces\`{a}ro mean. The simplest construction may be to choose a two-valued piecewise-constant density function on the half-line. To make the Ces\`{a}ro mean fluctuate, we need to alter the density along a rapid-increasing sequence, so that the constant density in the longest interval becomes the major contribution in the volume integral.

For each $n\ge0$, let
\[
 T_n=2^{2^n},\qquad
 a_n=\begin{cases}1,&n\ \text{even},\\2,&n\ \text{odd}.
 \end{cases}
\]
Define a piecewise-constant function $a:[0,\infty)\to[1,2]$ by
\[
 a(t)=\begin{cases}
 1,&0\le t<T_0,\\
 a_n,&T_n\le t<T_{n+1},\quad n\ge0,
 \end{cases}
\]
and let
\[
 w(x)=\begin{cases}1,&0\le x\le1,\\a(\log x),&x>1.
 \end{cases}
\]
Equip $X=[0,\infty)$ with the Euclidean metric and measure
$\dd\mu(x)=w(x)\dd x$.  In our construction, we regard $\alpha=1$.

\begin{theorem}\label{thm:counterexample}
The metric measure space $(X,|\cdot|,\mu)$ is $1$-Ahlfors
regular.  For $u(x)=(1-x)_+$ and every $1\le p<\infty$, there are two decreasing sequences tending to 0 along
which $\E_s(u)$ converges respectively to $\frac{2}{p(p+1)}$ and $\frac{4}{p(p+1)}$.
Consequently, neither the \MS{} limit nor the logarithmic-volume mean
\eqref{eq:logarithmic-mean} exists.
\end{theorem}

\begin{proof}
We first show the 1-Ahlfors regularity. Indeed, for $x\ge0$ and $r>0$,
\begin{align*}
	B_r(x)=[\max\{0,x-r\},x+r].
\end{align*}
Since the density function satisfies $w\in [1,2]$, we obtain the Ahlfors regularity estimate
\begin{equation}\label{eq:ray-ahlfors}
 r\le\mu(B_r(x))\le4r
 \qquad(x\ge0,\ r>0).
\end{equation}

Next, we show that $u\in\D_{1,p}$. Fix $s_0\in(0,1)$. Recall that
$u$ is supported in $[0,1]$, $0\le u\le1$, and
$|u(x)-u(y)|\le|x-y|$. Also $w\le2$, so
$w(x)w(y)\le4$. To estimate $[u]_{s_0,p;1}^p$, we split the integral into two parts: $0<|x-y|\le1$ and
$|x-y|>1$. Using $|u(x)-u(y)|\le|x-y|$ and $w(x)w(y)\le4$,
\[
\begin{aligned}
	&\iint_{0<|x-y|\le1}
	\frac{|u(x)-u(y)|^p}{|x-y|^{1+s_0p}}\dd\mu(x)\dd\mu(y)\\
	&\le 4\int_0^1\int_{\substack{y\ge0\\0<|x-y|\le1}}
	|x-y|^{p(1-s_0)-1}\dd y\dd x < \infty,
\end{aligned}
\]
because $p(1-s_0)>0$, the singularity at $x=y$ is integrable
and the domain is bounded.

Since $u$ vanishes outside $[0,1]$, we have
$|u(x)-u(y)|\le1$. Hence
\[
\begin{aligned}
	&\iint_{|x-y|>1}
	\frac{|u(x)-u(y)|^p}{|x-y|^{1+s_0p}}\dd\mu(x)\dd\mu(y)\\
	&\le 4\int_0^1\int_{x+1}^\infty
	(y-x)^{-1-s_0p}\dd y\dd x
	= \frac{4}{s_0p} < \infty.
\end{aligned}
\]
Combining both estimates yields
$[u]_{s_0,p;1}^p<\infty$, and therefore $u\in\D_{1,p}$.

Finally, we construct two desired sequences $s_n$ corresponding to even and odd integers $n$, respectively, as follows:
\[
 \varepsilon_n:=(T_nT_{n+1})^{-1/2},
 \qquad s_n:=\frac{\varepsilon_n}{p}.
\]
With the notions in \eqref{eq:scalar-tail}, we claim that \begin{equation}\label{eq:tail-subsequences}
 H_{s_n}(0)\longrightarrow\frac{a_n}{p}.
\end{equation}
Indeed, the substitution $x=e^t$ gives
\begin{align}
 H_s(0)
 &=s\int_1^\infty x^{-1-sp}w(x)\dd x=s\int_0^\infty e^{-spt}a(t)\dd t
 =\frac1p A_{sp}(a),
 \label{eq:weighted-tail}
\end{align}
where $A_{sp}(\cdot)$ is defined in \eqref{e.Abel}. Recall that $a=a_n$ on $[T_n,T_{n+1})$ and
$|a-a_n|\le1$ for all other cases. We know that
\begin{align*}
 |A_{\varepsilon_n}(a)-a_n|
 &\le\varepsilon_n\int_0^{T_n}e^{-\varepsilon_nt}\dd t
 +\varepsilon_n\int_{T_{n+1}}^\infty
 e^{-\varepsilon_nt}\dd t\\
 &=1-e^{-\varepsilon_nT_n}
   +e^{-\varepsilon_nT_{n+1}}\longrightarrow0
\end{align*}
by noting that $\varepsilon_nT_n=T_n^{-1/2}\longrightarrow0$ and
$\varepsilon_nT_{n+1}=T_n^{1/2}\longrightarrow\infty$. \eqref{eq:tail-subsequences} then follows from \eqref{eq:weighted-tail}.

Since $w(x)=1$ on $[0,1]$, we have
\begin{equation*}
 \norm{u}_p^p=\int_0^1(1-x)^p\dd x=\frac1{p+1}.
\end{equation*}
Therefore, Theorem~\ref{thm:main} (2) together with
\eqref{eq:tail-subsequences} implies
\[
 \E_{s_n}(u)-\frac{2a_n}{p(p+1)}\longrightarrow0.
\]
The proof is complete by using the definition of $a_n$ and  Theorem~\ref{thm:main} (3).
\end{proof}

\section{MS convergence on the Unbounded Sierpi\'nski Gasket}
\label{sec:sg-verification}

In this section we show that the unbounded Sierpi\'nski gasket provides a natural example for which Theorem~\ref{thm:main} applies.

\begin{figure}[tbph]
	\centering
	\subfloat[$K$]{\includegraphics[width=2cm]{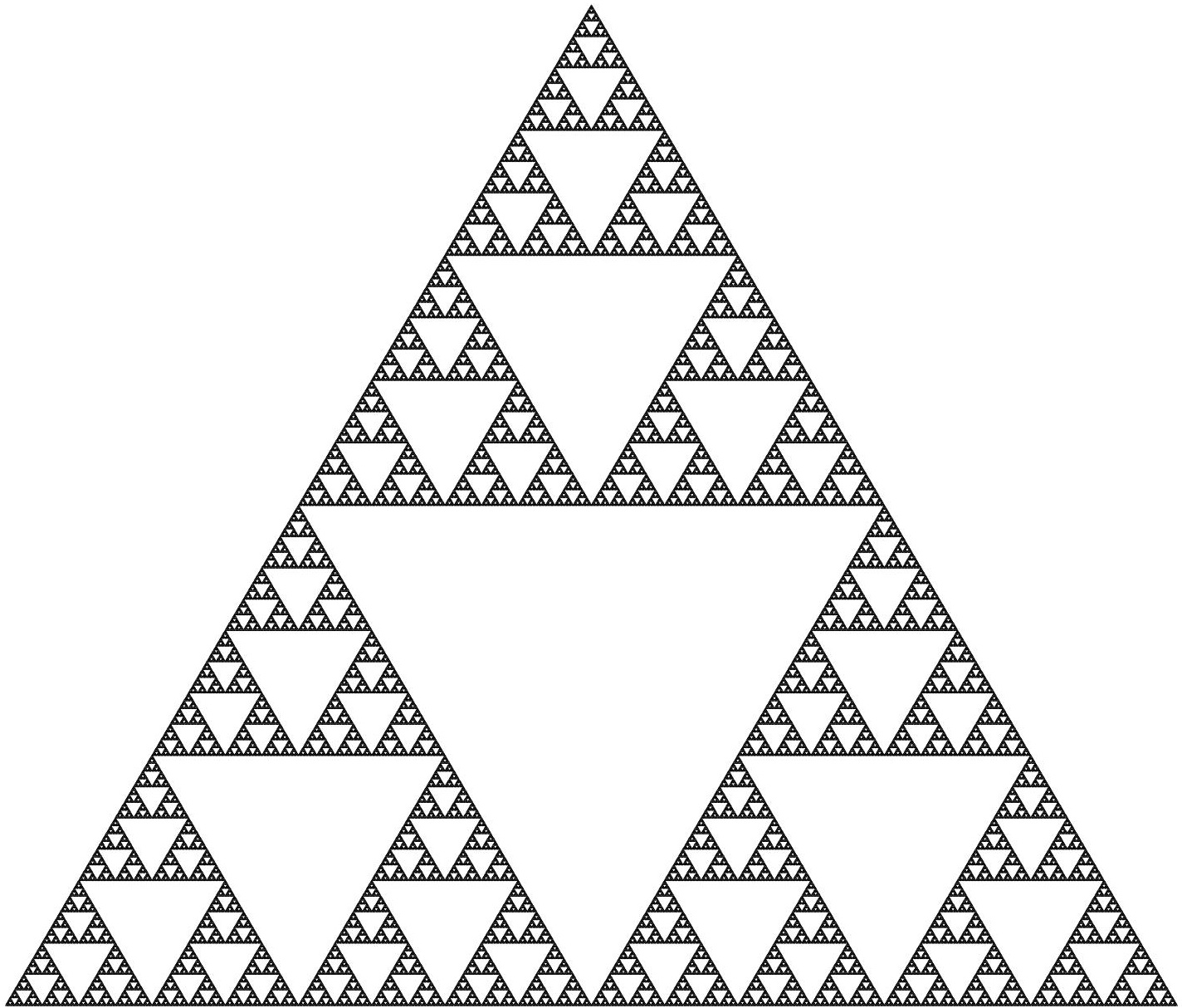}}\quad %
	\subfloat[$2K$]{\includegraphics[width=4cm]{gasket.jpg}}\quad
	\subfloat[$2^2K$]{\includegraphics[width=8cm]{gasket.jpg}}
	\caption{The unbounded Sierpi\'{n}ski gasket $X_{\mathrm{SG}}=\bigcup_{n\ge0}2^nK$ with $n=0,1,2$.}
	\label{fig1}
\end{figure}

Let $K$ be the standard Sierpi\'nski gasket generated by the IFS consisting of the three maps
$$
F_i(z)=\frac{z+p_i}{2},\qquad
p_0=(0,0),\quad p_1=(1,0),\quad p_2=(\frac{1}{2},\frac{\sqrt{3}}{2}).
$$
The Hausdorff dimension of $K$ is
$\alpha=\frac{\log3}{\log2}$, Define the unbounded Sierpi\'nski gasket (see Figure \ref{fig1}) by
$$
X_{\mathrm{SG}}=\bigcup_{n\ge0}2^nK,
$$
equipped with the Euclidean metric $d(x,y)=|x-y|$ and the $\alpha$-dimensional Hausdorff measure $\mu$. Fix the base point $o=p_0$ and write
$$
V_o(r)=\mu(B_r(o)),\qquad
h_o(t)=e^{-\alpha t}V_o(e^t),\qquad
M_o(T)=\frac1T\int_0^T h_o(t)\dd t.
$$
The space $(X_{\mathrm{SG}},d,\mu)$ is $\alpha$-Ahlfors regular (see, for example, \cite[Example 8.2]{Che26}); hence, the growth condition \eqref{eq:upper-growth} in Theorem~\ref{thm:main} is satisfied. The space \(\mathcal{D}_{\alpha,p}\) contains functions with positive $L^p$-norm, for example, the indicator function of the unit ball \(u = \chi_{B(o,1)}\).

Because of the global self-similarity, $D(z)=2z$ is a bijection of $X_{\mathrm{SG}}$ onto itself satisfying $D(o)=o$, and
\begin{align}
	V_o(2r)
	&=\mu(B_{2r}(p_0))
	=\mu(D B_r(p_0))
	=3V_o(r)
	=2^{\alpha} V_o(r).
\end{align}
Therefore
\begin{align*}
	h_o(t+\log2)
	&=e^{-\alpha(t+\log2)}V_o(e^{t+\log2})
	=e^{-\alpha t}2^{-\alpha}V_o(2e^t)\\
	&=e^{-\alpha t}2^{-\alpha}\cdot2^{\alpha} V_o(e^t)
	=h_o(t).
\end{align*}
It follows that
\begin{align*}
	M_o(T)
	&=\frac1T\int_0^T h_o(t)\dd t
	\longrightarrow
\frac{1}{\log2}\int_0^{\log2} h_o(t)\dd t.
\end{align*}
Hence condition (c) in Theorem~\ref{thm:main} (3) holds, and the logarithmic volume mean
\[
\av_{\alpha}(X_{\mathrm{SG}})
=\lim_{T\to\infty}M_o(T)
\]
exists and is finite. By Theorem~\ref{thm:main} (3), for every $u\in\mathcal{D}_{\alpha,p}$,
\begin{align}
	\lim_{s\downarrow0}\mathcal{E}_s(u)
	&=\frac{2\alpha}{p}\av_{\alpha}(X_{\mathrm{SG}})\|u\|_p^p.
\end{align}

\section*{Acknowledgement}
    The authors acknowledge the use of AI tools for language polishing and formatting. Jin Gao was supported by the National Natural Science
	Foundation of China (No. 12271282), and by the Zhejiang Provincial Natural Science
	Foundation of China (No. LQN25A010019). Zhenyu Yu was supported by the Natural
	Science Foundation of Hunan Province, China (No. 2025JJ60039) and by the National University of Defense Technology (No. ZK25-05).


\begin{thebibliography}{99}
\bibitem{BBM}
{\sc J.~Bourgain, H.~Brezis, and P.~Mironescu}, {\em Another look at Sobolev
	spaces, Optimal Control and Partial Differential Equations (J.L. Menaldi et
	al. eds)} (IOS Press, Amsterdam, 2001).

\bibitem{BSY23}
D. Brazke, A. Schikorra and P.-L. Yung, Bourgain-Brezis-Mironescu convergence via Triebel-Lizorkin spaces, \emph{Calc. Var. Partial Differ. Equ.} \textbf{62} (2023), Paper No.~41.

\bibitem{Che26}
A. Chen, Stability of heat kernel bounds under pointed Gromov-Hausdorff convergence, \emph{J. Funct. Anal.} \textbf{291} (2026), 111488.

\bibitem{DGPYYZ24}
F. Dai, L. Grafakos, Z. Pan, D. Yang, W. Yuan and Y. Zhang, The Bourgain-Brezis-Mironescu formula on ball Banach function spaces, \emph{Math. Ann.} \textbf{388} (2024), 1691--1768.

\bibitem{DDFGP}
E. Davoli, G. Di Fratta, R. Giorgio and A. Pinamonti, Necessary and sufficient conditions for the Maz'ya--Shaposhnikova formula in (fractional) Sobolev spaces, preprint, 2025, arXiv:2509.23226.

\bibitem{GYZ22}
J. Gao, Z. Yu and J. Zhang, Convergence of Dirichlet forms and Besov norms on scale irregular Sierpi\'nski gaskets, \emph{Fractals} \textbf{30} (2022), Paper No.~2250163.

\bibitem{GYZ23}
J. Gao, Z. Yu and J. Zhang, Convergence of p-energy forms on homogeneous p.c.f self-similar sets, \emph{Potential Anal.} \textbf{59} (2023), 1851--1874.

\bibitem{GYZ26}
J. Gao, Z. Yu and J. Zhang, Heat kernel-based p-energy norms on metric measure spaces, \emph{Math. Proc. Camb. Phil. Soc.} \textbf{180} (2026), 685--728.

\bibitem{GL20A}
Q. Gu and K.-S. Lau, Dirichlet forms and convergence of Besov norms on self-similar sets, \emph{Ann. Acad. Sci. Fenn. Math.} \textbf{45} (2020), 625--646.

\bibitem{GuYang}
Q. Gu and P.-L. Yung, A new formula for the $L^p$ norm, \emph{J. Funct. Anal.} \textbf{281} (2021), Paper No.~109075.

\bibitem{Han}
B.-X. Han, On the asymptotic behaviour of the fractional Sobolev seminorms: A geometric approach, \emph{J. Funct. Anal.} \textbf{287} (2024), Paper No.~110608.

\bibitem{HPXZ}
B.-X. Han, A. Pinamonti, Z. Xu and K. Zambanini, Maz'ya--Shaposhnikova meet Bishop--Gromov, \emph{Potential Anal.} \textbf{63} (2025), 513--529.

\bibitem{LPZ24}
P. Lahti, A. Pinamonti and X. Zhou, A characterization of BV and Sobolev functions via nonlocal functionals in metric spaces, \emph{Nonlinear Anal.} \textbf{241} (2024), 113467.

\bibitem{MS}
V. Maz'ya and T. Shaposhnikova, On the Bourgain, Brezis, and Mironescu theorem concerning limiting embeddings of fractional Sobolev spaces, \emph{J. Funct. Anal.} \textbf{195} (2002), 230--238; erratum, \emph{J. Funct. Anal.} \textbf{201} (2003), 298--300.

\bibitem{PP08}
K. Pietruska-Pa{\l}uba, Limiting behaviour of Dirichlet forms for stable processes on metric spaces, \emph{Bull. Pol. Acad. Sci. Math.} \textbf{56} (2008), 257--266.

\bibitem{Shi25}
R. Shimizu, Characterizations of Sobolev functions via Besov-type energy functionals in fractals, \emph{Potential Anal.} \textbf{63} (2025), 2121--2156.

\bibitem{Yang}
M. Yang, Equivalent semi-norms of non-local Dirichlet forms on the Sierpi\'nski gasket and applications, \emph{Potential Anal.} \textbf{49} (2018), 287--308.


\end{thebibliography}
\end{document}